\documentclass[12pt, reqno]{amsart}
\usepackage{epsfig,fullpage,subfig} 

\usepackage{dsfont}

\usepackage{amsmath}
\usepackage{amssymb}
\usepackage{verbatim}
\usepackage{enumerate}
\usepackage[all]{xy}
\usepackage[mathscr]{eucal}
\usepackage{amsthm}
\usepackage{amscd}
\usepackage{amsfonts}
\usepackage{latexsym}
\usepackage{graphicx}

\newtheorem{theorem}{Theorem}[section]
\newtheorem{lemma}[theorem]{Lemma}
\newtheorem{proposition}[theorem]{Proposition}
\newtheorem{corollary}[theorem]{Corollary} 
\theoremstyle{definition}  
\newtheorem{definition}[theorem]{Definition}
\newtheorem{example}[theorem]{Example}

\newtheorem{remark}[theorem]{Remark}

\newcommand{\Tr}{\text{Tr}}
\newcommand{\id}{\operatorname{Id}}

\newcommand{\End}{\operatorname{End}}

\newcommand{\Hom}{\operatorname{Hom}} 
 
\def\uRep{\underline{\operatorname{Re}}\!\operatorname{p}}

\newcommand{\Rep}{\operatorname{Rep}}

\newcommand{\tr}{\text{tr}\,}

\newcommand{\ev}{\operatorname{ev}}
\newcommand{\coev}{\operatorname{coev}}

\newcommand{\Lift}{\operatorname{Lift}}

\newcommand{\F}{\mathcal{F}}

\newcommand{\mN}{{\mathcal N}}

\newcommand{\Z}{\mathbb{Z}}
\newcommand{\cat}{\mathcal{C}}

\def\gruRep{\operatorname{g}\!\underline{\operatorname{rRe}}\!\operatorname{p}} 
\newcommand{\sgn}{\operatorname{sgn}}
\newcommand{\ideal}{\mathcal{I}}
\newcommand{\ob}{\operatorname{Ob}(\cat)}

\newcommand{\unit}{\ensuremath{\mathds{1}}}
\newcommand{\Id}{\operatorname{Id}}
\newcommand{\md}{\operatorname{\mathsf{d}}}
\newcommand{\mt}{\operatorname{\mathsf{t}}}
\renewcommand{\tr}{\operatorname{tr}}
\renewcommand{\Tr}{\operatorname{Tr}}

\begin{document}

\title{Modified Traces on Deligne's Category $\uRep (S_{t})$}

\author{Jonathan Comes }
\address{Department of Mathematics \\
          University of Oregon\\
          Eugene, OR 97403}
\email{jcomes@uoregon.edu}
\author{Jonathan R. Kujawa}
\address{Department of Mathematics \\
          University of Oklahoma \\
          Norman, OK 73019}
\thanks{Research of the second author was partially supported by NSF grant
DMS-0734226 and NSA grant H98230-11-1-0127.}\
\email{kujawa@math.ou.edu}
\date{\today}

\begin{abstract} Deligne has defined a category which interpolates among the representations of the various symmetric groups.  In this paper we show Deligne's category admits a unique nontrivial family of modified trace functions.  Such modified trace functions have already proven to be interesting in both low-dimensional topology and representation theory.  We also introduce a graded variant of Deligne's category, lift the modified trace functions to the graded setting, and use them to recover the well-known invariant of framed knots known as the writhe. 
\end{abstract}

\maketitle

\section{Introduction}\label{S:Intro}  

\subsection{}  Let $F$ denote a field of characteristic zero and let $t \in F$.  Recently Deligne gave a definition of a category, $\uRep (S_{t})$, which interpolates among the representations over $F$ of the various symmetric groups \cite{Del07}.  Somewhat more precisely: when $t$ is \emph{not} a nonnegative integer, the category $\uRep (S_{t})$ is semisimple and when $t$ is a nonnegative integer, then a natural quotient of $\uRep (S_{t})$ is equivalent to the category of representations over $F$ of the symmetric group on $t$ letters.

Axiomatizing Deligne's construction, Knop gave a number of additional examples of interpolating categories, including representations of finite general linear groups and of wreath products \cite{Knop1, Knop2}.  More recently Etingof defined interpolating categories in other settings which include degenerate affine Hecke algebras and rational Cherednik algebras \cite{Etingof}.  Most recently Mathew provided an algebro-geometric setup for studying these categories when the parameter is generic \cite{Mathew}.  Comes and Wilson study Deligne's analogously defined $\uRep (GL_{t})$ and use it to completely describe the indecomposable summands of tensor products of the natural module and its dual for general linear supergroups \cite{CW}.   

We will be interested in Deligne's $\uRep(S_t)$.  Besides motivating the new direction of research in representation theory discussed above, it is an object of study in its own right.  Comes and Ostrik completely describe the indecomposable objects and blocks in $\uRep (S_{t})$ in \cite{CO1}, and classify tensor ideals along the way to proving a conjecture of Deligne  in \cite{CO2}.  Recently, Del Padrone used Deligne's category to answer several questions which arose out of the work of Kahn in studying the rationality of certain zeta functions \cite{Del}.

\subsection{} In this paper we will be interested in the tensor and duality structure of Deligne's category.  It is well understood that categories with a tensor and duality structure play an important role in low-dimensional topology.  The basic idea is to start with some suitable category (called a \emph{ribbon category}) which admits a tensor product and braiding isomorphisms
\[
c_{V,W} : V \otimes W \to W \otimes V.
\]  for all V and W in the category.  One uses the category to create invariants of knots, links, $3$-manifolds, etc.\  by interpreting the relevant knot or link as a morphism in the category using the braiding to represent crossings in the knot or link diagram.  See, for example,  \cite{BK, Kas, Tu} where these constructions are made precise.  

A reoccurring difficulty in this approach are the objects with categorical dimension zero.  These objects necessarily give trivial topological invariants.  Tackling this problem Geer and Patureau-Mirand defined modified trace and dimension functions for typical representations of quantum groups associated to Lie superalgebras \cite{GP}.  With Turaev they generalized this construction to include, for example, the quantum group for $\mathfrak{sl}(2)$ at a root of unity \cite{GPT}.  Along with various coauthors, they have gone on to vastly generalize their construction and use it to obtain new topological invariants.  In particular, they have shown how to use modified traces to give generalized Kashaev and Turaev-Viro-type 3-manifold invariants, to show that these invariants coincide, and that they extend to a relative Homotopy Quantum Field Theory.  Especially intriguing, they also show how to use this theory to generalize the quantum dilogarithmic invariant of links appearing in the well-known Volume Conjecture.  See \cite{GP2, GKT} and references therein.  

On the algebra side of the picture the second author worked jointly with Geer and Patureau-Mirand to provide a ribbon categorical framework for modified trace and dimension functions and considered a number of examples coming from representation theory \cite{GKP}.   They showed that these functions generalize well known results from representation theory as well as giving entirely new insights.  For example, this point of view leads to a natural generalization of a conjecture  by Kac and Wakimoto for complex Lie superalgebras.   Recently Serganova proved the original Kac-Wakimoto conjecture for the basic classical Lie superalgebras and the generalized Kac-Wakimoto conjecture for $\mathfrak{gl}(m|n)$ \cite{Ser}.  The generalized Kac-Wakimoto conjecture is in turn used to compute the complexity of the finite dimensional simple supermodules for $\mathfrak{gl}(m|n)$ by Boe, Nakano, and the second author \cite{BKN4}.

Despite the success of this program, it remains mysterious when these modified dimension functions exist.  In \cite{GKP} the authors provide examples which show that rather elementary categories in representation theory (e.g.\ certain representations of the Lie algebra $\mathfrak{sl}_{2}(k)$ over field of characteristic $p$) can fail to have modified dimensions.   Motivated by this gap in our understanding and by the aforementioned applications within low-dimensional topology and representation theory, in this paper we investigate modified trace and dimension functions within Deligne's category $\uRep (S_{t})$.  

\subsection{} Our main result (Theorem~\ref{amb}) proves that when $t$ is a nonnegative integer the only nontrivial ideal in $\uRep(S_t)$ always admits a modified trace.  This is remarkable as all known existence proofs and all previous examples of modified traces involve an abelian category.  Deligne's category, which is only abelian when $t$ is not a nonnegative integer, provides the first example of a nonabelian ribbon category which admits modified traces.  

A second interesting outcome of our investigation is the following observation.  In \cite{GKP} if $\cat$ is a $F$-linear category and $X$ is an object with $\End_{\cat}(X)/\operatorname{Rad}\left(\End_{\cat}(X) \right) \cong F$, then $X$ is called \emph{ambidextrous} if the canonical map 
\[
\End_{\cat}(X) \to \End_{\cat}(X)/\operatorname{Rad}\left(\End_{\cat}(X) \right) \cong F
\] defines a modified trace function.  In \emph{loc.\ cit.\ }many results about modified trace and dimension functions are most naturally stated for ambidextrous objects.  One might then expect that if an object with a local endomorphism ring admits a modified trace function that it should be the canonical map and, hence, $X$ should be ambidextrous.  Remarkably, we see this is not the case (see Section~\ref{0trace}).  This illustrates the subtlety of the theory.

\subsection{}  Aside from vanishing categorical dimension, the second main obstacle to using a ribbon category to construct nontrivial topological invariants is when the category has a \emph{symmetric braiding}.  That is,  
\[
c_{W,V} \circ c_{V,W} = \Id_{V \otimes W}
\] for all objects $V$ and $W$.  Knot theoretically, this corresponds to over- and under-crossings being equal.  Such categories yield only trivial topological invariants.  Deligne's category has a symmetric braiding.  This motivates the search for categories with nonsymmetric braiding arising from $\uRep (S_{t})$.

In Section~\ref{SS:gradedvariant} we define a graded variant of Deligne's category, $\gruRep (S_{t})_{q}$, and prove that there is a ``degrading'' functor $\F: \gruRep (S_{t})_{q} \to \uRep (S_{t})$.  Using this functor we can lift the modified trace functions on $\uRep (S_{t})$ to $\gruRep (S_{t})_{q}$.  In particular, the graded category has a nonsymmetric braiding and the modified trace function defines a nontrivial knot invariant.  In this way we can use Deligne's category to recover the well-known invariant of framed knots known as the \emph{writhe}.  

\subsection{}  The results of this paper raise a number of intriguing questions.  As mentioned above, Deligne's construction naturally generalizes to a wide variety of settings within representation theory.  We expect that modified traces should exist for many of these other categories and it would be interesting to investigate this question. Deligne's category $\uRep (S_{t})$ has a relatively elementary structure (for example, it has a single nontrivial tensor ideal) and we expect that studying modified traces in these other settings will be significantly more involved.  

\subsection{Acknowledgements}  This paper began during a fortuitous meeting of the authors at a workshop on Combinatorial Representation Theory in March 2010 at the Mathematics Forschungsinstitut Oberwolfach in Germany.  We are grateful to the staff of the Institute for an excellent working environment and to the organizers for arranging a stimulating workshop.  The majority of the work on this paper was done while the first author was enjoying a position at the Technische Universit\"at M\"unchen; he would like to thank the university for providing an excellent research environment.  The second author would also like to thank Nathan Geer and Bertrand Patureau-Mirand for many stimulating conversations. 

\section{Ribbon Categories and Traces}\label{S:ribboncategories}   The authors of \cite{GKP} define modified trace functions for ideals in ribbon categories.  In this section we give a brief overview of this theory but refer the reader to the above paper for further details and proofs. 

\subsection{Ribbon Categories}\label{SS:ribboncats}  For notation and the general setup of ribbon categories our references are \cite{Tu} and \cite{Kas}.  A \emph{tensor category} $\cat$ is a category equipped with a covariant
bifunctor 
\[
\otimes :\cat \times \cat\rightarrow \cat
\]
called the tensor product, a unit object $\unit$, an associativity constraint, and left and
right unit constraints such that the Triangle and Pentagon Axioms hold (see \cite[XI.2]{Kas}).  In particular, for any $V$ in $\cat$, $\unit \otimes V$ and $V \otimes \unit$ are canonically isomorphic to $V$.

A \emph{braiding} on a tensor category $\cat$ consists of a family of isomorphisms 
\[
\{c_{V,W}: V \otimes W \rightarrow W\otimes V \},
\]
 defined for each pair of objects $V,W$ which satisfy the Hexagon Axiom \cite[XIII.1 (1.3-1.4)]{Kas} as well as the naturality condition expressed in the commutative diagram \cite[(XIII.1.2)]{Kas}.
We say a tensor category is \emph{braided} if it has a braiding.  We call the braiding \emph{symmetric} if 
\[
c_{W, V} \circ c_{V,W} = \Id_{V \otimes W}
\]
for all $V$ and $W$ in $\cat$.

A tensor category $\cat$ has \emph{duality} if for each object $V$ in $\cat$ there exits an object $V^{*}$ and coevaluation and evaluation morphisms\footnote{In \cite{GKP} these maps are denoted $b_{V}$ and $d_{V}$, respectively.} 
$$\coev_{V}: \unit \rightarrow V\otimes V^{*} \: \text{ and } \: \ev_{V}: V^{*}\otimes V \rightarrow \unit$$
satisfying relations \cite[XIV.2 (2.1)]{Kas}.

A \emph{twist} in a braided tensor category $\cat$ with duality is a family 
\[
\{ \theta_{V}:V\rightarrow V \}
\]
of natural isomorphisms defined for each object $V$ of $\cat$ satisfying relations \cite[(XIV.3.1-3.2)]{Kas}. Let us point out that the existence of twists is equivalent to having functorial isomorphisms $V \xrightarrow{\cong} V^{**}$ for all $V$ in $\cat$ (cf.\ \cite[Section 2.2]{BK}).

A \emph{ribbon category} is a braided tensor category with duality and twists.  A fundamental feature of ribbon categories is the fact that morphisms in the category can be represented diagrammatically and that isotopic diagrams correspond to equal morphisms.  For the sake of brevity, we do not give the graphical calculus here but encourage the interested reader to refer to \cite{Kas}. 

In a ribbon category it is convenient to also define the morphisms

\begin{equation*}
\coev'_{V} : \unit \to V^{*} \otimes V \text{ and } \ev'_{V}: V \otimes V^{*} \to \unit
\end{equation*}
which are given by
\begin{equation*}
\coev'_V =(\Id_{V^*}\otimes \theta_V) \circ c_{V,V^*} \circ \coev_{V}  \text{ and } \ev'_V  = \ev_{V} \circ c_{V,V^*} \circ (\theta_V\otimes \Id_{V^*}).
\end{equation*}
Finally, the \emph{ground ring} of a ribbon category $\cat$ is 
\[
K=\End_{\cat}(\unit ).
\]  We assume $K$ is a field, that the category is $K$-linear, and that the tensor product is bilinear. Later references to linearity will always be with respect to $K$.
Ultimately the ground ring will be a fixed field $F$ of characteristic zero and the categories in question will be $F$-linear.

\subsection{Ideals in $\cat$}\label{SS:ideals} There are two closely related notions of an ideal within a ribbon category.  The first we discuss is used in \cite{GKP} and defined via objects.  We discuss the second notion in Section~\ref{SS:homideals}.   Note that here and elsewhere if $f$ and $g$ are morphisms, then we write $fg$ for the composition $f \circ g$.

\begin{definition}\label{D:ideal} We say a full subcategory $\ideal$ of a ribbon category $\cat$ is an \emph{ideal} if the following two conditions are met:
\begin{enumerate}
\item  If $V$ is an object of $\ideal$ and $W$ is any object of $\cat$, then $V\otimes W$ is an object of $\ideal$.
\item  $\ideal$ is closed under retracts; that is, if $V$ is an object of $\ideal$, $W$ an object of $\cat$, and if there exists morphisms $f:W \rightarrow V$, $g:V\rightarrow W$ such that $gf=\Id_W$, then $W$ is an object of $\ideal$.
\end{enumerate}
\end{definition} Trivially, if $\ideal$ consists of just the zero object or $\ideal = \cat$, then $\ideal$ is an ideal of the category.  We say an ideal $\ideal$ is a \emph{proper} ideal if it contains a nonzero object and is not all of $\cat$.

\subsection{Traces in Ribbon Categories}\label{SS:traces} For any objects $V,W$ of $\cat$ and $f \in \End_{\cat}( V \otimes W)$, set
\begin{equation}\label{E:trL}
\tr_{L}(f)=(\ev_{V}\otimes \Id_{W})(\Id_{V^{*}}\otimes
f)(\coev'_{V}\otimes \Id_{W}) \in \End_{\cat}(W),
\end{equation} and
\begin{equation}\label{E:trR}
\tr_{R}(f)=(\Id_{V}\otimes \ev'_{W})  (f \otimes \Id_{W^{*}})
(\Id_{V}\otimes \coev _{W}) \in \End_{\cat}(V).
\end{equation}

\begin{definition}\label{D:trace}  If $\ideal$ is an ideal in $\cat$, then a \emph{trace on $\ideal$} is a family of linear functions
$$\mt = \{\mt_V:\End_\cat(V)\rightarrow K\}$$
where $V$ runs over all objects of $\ideal$ and such that following two conditions hold:
\begin{enumerate}
\item  If $U\in \ideal$ and $W\in \ob$ then for any $f\in \End_\cat(U\otimes W)$ we have
\begin{equation}\label{E:VW}
\mt_{U\otimes W}\left(f \right)=\mt_U \left( \tr_R(f)\right).
\end{equation}
\item  If $U,V\in \ideal$ then for any morphisms $f:V\rightarrow U $ and $g:U\rightarrow V$  in $\cat$ we have 
\begin{equation}\label{E:fggf}
\mt_V(g f)=\mt_U(f  g).
\end{equation} 
\end{enumerate}
\end{definition}
Using the trace on $\ideal$ introduced above, we define a modified
dimension function on objects in $\ideal$.  Namely, we define the \emph{modified dimension function} 
\[
\md_{\mt}: \operatorname{Ob}(\ideal ) \to  K
\] by the formula 
\[
\md_{\mt} (V) = \mt_{V}\left(\Id_{V} \right).
\]

\begin{example}\label{X:categoricaltrace}
If $\cat$ is a ribbon category, then $\cat$ itself is an ideal and the well known \emph{categorical trace} function
\[
\tr_{\cat}: \End_{\cat}(V) \to  K
\] given by 
\[
\tr_{\cat}(f) = \ev'_{V}  (f \otimes 1)  \coev_{V}
\] defines a trace on $\cat$.  The modified dimension function then coincides with the familiar \emph{categorical dimension} function.
\end{example}

The following theorem from \cite{GKP} gives a convenient way of creating ideals with traces.  Assume that $J$ in $\cat$ admits a linear map 
\[
\mt_{J}: \End_{\cat}(J) \to K
\] which satisfies 
\begin{equation*}
\mt_{J}\left( \tr_{L}(h) \right) = \mt_{J} \left( \tr_{R}(h)\right),
\end{equation*} for all $h \in \End_{\cat}(J\otimes J)$. 
Such a linear map is called an \emph{ambidextrous trace on $J$}. 

For an object $J$, let $\ideal_{J}$ denote the ideal whose objects are all objects which are retracts of $J \otimes X$ for some $X$ in $\cat$.
\begin{theorem}\label{T:ambitrace} If $J$ is an object of $\cat$ which admits an
  ambidextrous trace, then there is a unique trace on $\ideal_{J}$ determined
  by that ambidextrous trace.
\end{theorem}

\subsection{Tensor Ideals in a Ribbon Category}\label{SS:homideals}
A somewhat different notion of ideal is used in \cite{CO1, CO2}.  As we need both, we define it here and discuss the relationship with the earlier definition.  To distinguish the two we call these tensor ideals.  They are defined via morphisms as follows.

\begin{definition}\label{D:Homideal}  A \emph{tensor ideal}, $J$, of $\cat$ is a family of subspaces 
\[
J(X,Y) \subseteq \Hom_{\cat}(X,Y)
\] for all pairs of objects $X,Y$ in $\cat$ subject to the following two conditions:
\begin{enumerate}
\item $ghk \in I(X,W)$ for each $k \in \Hom_{\cat}(X,Y)$, $h \in  J(Y,Z)$, and $g \in \Hom_{\cat}(Z,W)$.
\item $g\otimes \Id_{Z} \in J(X\otimes Z, Y\otimes Z)$ for every object $Z$ and every $g \in J(X,Y)$. 
\end{enumerate}
\end{definition} Trivially, for every pair of objects $X$ and $Y$ one can take $J(X,Y) =0$ and obtain a tensor ideal; similarly, for every pair of objects one can take $J(X,Y) = \Hom_{\cat}(X,Y)$.  A tensor ideal $J$ is called \emph{proper} if $J(X,Y)$ is a proper nonzero subspace of $\Hom_{\cat}(X,Y)$ for at least one pair of objects $X$ and $Y$ in $\cat$.

\subsection{}  If $\ideal$ is an ideal of $\cat$ in the sense of Definition~\ref{D:ideal}, then one can define subspaces 
\[
J(X,Y) = \left\{f \in \Hom_{\cat}(X,Y) \mid \text{there exists $Z$ in $\ideal$,  $g: X \to Z$, $h:Z \to Y$ so that $f=hg$} \right\}.
\]  Then $J$ forms a tensor ideal and we write $J(\ideal)$ for this tensor ideal.

Conversely, if $J$ is a tensor ideal, then one can define $\ideal$ to be the full subcategory consisting of 
all objects $V$ in $\cat$ such that $\Id_{V} \in J(V,V)$.  This is an ideal of $\cat$  and we write $\ideal (J)$ for this ideal.  

In the following lemma we record the basic properties relating these two notions of an ideal.  The proofs are elementary arguments using the definitions and previous parts of the lemma.

\begin{lemma}\label{L:ideals}  Let $\cat$ be a ribbon category.
\begin{enumerate}
\item If $\ideal$ is an ideal of $\cat$, then $\ideal = \ideal(J(\ideal))$.
\item If $J$ is a tensor ideal of $\cat$, then $J(\ideal(J)) \subseteq J$.  That is, 
\[
 J(\ideal(J))(X,Y) \subseteq J(X,Y) 
\]
for all pairs of objects $X,Y$.
\item The ideal $\ideal$ is the zero ideal if and only if $J(\ideal )$ is the zero tensor ideal.
\item The ideal $\ideal$ is the entire category $\cat$ if and only if 
\[
J(\ideal)(X,Y)=\Hom_{\cat}(X,Y)
\]
for all pairs of objects $X,Y$ in $\cat$. 
\item If $\cat$ has a unique proper tensor ideal, say $J$, and $\ideal$ is a proper ideal of $\cat$, then $\ideal$ is the unique proper ideal and $\ideal = \ideal (J)$.
\end{enumerate}
\end{lemma}

\subsection{} A fundamental example of a tensor ideal is the so-called negligible morphisms.  Namely, let $\cat$ be a ribbon category and call a morphism $g: X \to Y$ \emph{negligible} if for all $h \in \Hom_{\cat}(Y,X)$, one has 
\[
\tr_{\cat}(gh) =0,
\] where $\tr_{\cat}$ denotes the categorical trace.  Setting $\mN(X,Y)$ to be the subspace of $\Hom_{\cat}(X,Y)$ of all negligible morphisms, one can check that $\mN$ is a tensor ideal. For short we call an object \emph{negligible} if it is an object in $\ideal (\mN )$.

When $t \in \Z_{\geq 0}$, $\mN$ is a proper tensor ideal of Deligne's category $\uRep(S_{t})$.  The quotient of $\uRep (S_{t})$ by this tensor ideal is equivalent to the category of finite dimensional representations over $F$ of the symmetric group $S_{t}$ (see \cite[Theorem 3.24]{CO1}).  It is in this sense that Deligne's category interpolates among the representations of the various symmetric groups.

\section{Deligne's Category $\uRep(S_t)$ }\label{}

Fix a field $F$ of characteristic zero and fix $t\in F$.   For $n \geq 0$ we write $P_{n}$ for the set of partition diagrams with vertex set $\{1,\dotsc , n,1', \dotsc , n' \}$ and $FP_{n}=FP_{n}(t)$ for the partition algebra spanned by $P_{n}$ with parameter $t \in F$.  In particular, note that the symmetric group on $n$ letters, $S_{n}$, can canonically be identified with a subset of $P_{n}$ and, moreover, the group algebra $FS_{n}$ can be identified as a subalgebra of $FP_{n}$.  We use this identification without comment in what follows.  More generally, for $a,b \in \Z_{\geq 0}$ we write $FP_{a,b}=FP_{a,b}(t)$ for the vector space spanned by the partition diagrams with vertex set $\{1, \dotsc , a, 1', \dotsc , b' \}$.

Following the notation in \cite{CO1}, we write $\uRep(S_t; F)=\uRep(S_t)$ for the category defined by Deligne which interpolates among the representations of the symmetric groups.  This is an additive (not necessarily abelian) ribbon category with a symmetric braiding.  For a precise definition of $\uRep (S_{t})$ and its ribbon category structure, we refer the reader to \cite[Section 2.2]{CO1}.  Regardless of $t$, the isomorphism classes of indecomposable objects in $\uRep(S_t)$ are in bijective correspondence with Young diagrams of arbitrary size (see \cite[Theorem 3.7]{CO1}).  Following \emph{loc.~cit.}, we will write $L(\lambda)$ for the indecomposable object (defined up to isomorphism) in $\uRep(S_t)$ corresponding to Young diagram $\lambda$.

To avoid potential confusion, it is important to point out that morphisms in $\uRep (S_{t})$ are given by pictures which are a priori unrelated to the graphical calculus of ribbon categories.  More precisely, the morphisms in $\uRep (S_{t})$ are linear combinations of so-called partition diagrams and, as such, are usually given via pictures.  We follow this convention in what follows.  Fortunately, the pictures which represent the morphisms of a ribbon category (e.g.\ the evaluation, coevaluation, and braiding morphisms) are very similar to the pictures for these morphisms in the graphical calculus of ribbon categories. And the rules for tensor product and composition (horizontal and vertical concatenation, respectively) are the same in both settings\footnote{Note, however, that we follow \cite{CO1} and ``compose down the page'' so that the diagram for $fg$ has $g$ placed atop $f$.}. The differences between the two graphical settings are minor and, consequently, the reader should not have any difficulty using context to make clear what is meant in what follows.

\section{A trace on the ideal of negligibles in $\uRep(S_0)$}  \label{0trace}

We first work out the easiest example when the category is  $\uRep(S_0)$.  In this case everything can be computed explicitly.

\subsection{Defining the Trace Function} 

Consider the indecomposable object $L(\Box)$ in the category $\uRep(S_0)$.  We will define a trace on the ideal $\ideal_{L(\Box)}$ by verifying by explicit computation that $L(\Box)$ admits an ambidextrous trace.  

 In order to define such a trace, we study the endomorphisms of the object $L(\Box)\otimes L(\Box)$.
By \cite[Proposition 6.1]{CO1} we can identify $\End(L(\Box))$ with the partition algebra $FP_1(0)$.  Hence $\End(L(\Box)\otimes L(\Box))=FP_2(0)$.  Consider the following table:
$$\includegraphics{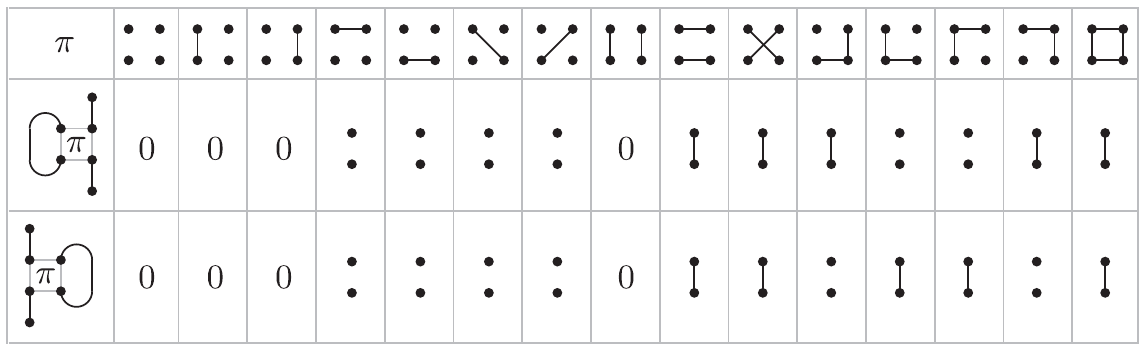}$$
From the table above we have the following:  A linear map $\mt:\End(L(\Box))\to F$ satisfies 
$$\includegraphics{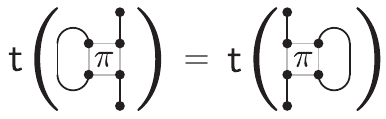}$$ for all partition diagrams $\pi\in FP_2(0)$ if and only if $\mt$ is constant on the two partition diagrams in $FP_1(0)$.  Therefore there is a unique ambidextrous trace function for $L(\Box)$ up to a constant multiple.  Hence by Theorem~\ref{T:ambitrace}  there is a unique trace on $\ideal_{L(\Box)}$ up to constant multiple.  We normalize by setting $\mt$ to be the trace function with $\mt(\id_{L(\Box)})=1$.  In summary, we have the following result.

\begin{theorem}\label{T:tzero}  There is a unique trace $\mt = \left\{\mt_{V} \right\}_{V \in \ideal_{L(\Box)}}$ on $\ideal_{L(\Box)}$ such that $\mt_{L(\Box)}(\id_{L(\Box)} )=1$.
\end{theorem}

We note that by the classification of tensor ideals in $\uRep(S_0)$ in \cite{CO2} there is a unique proper tensor ideal in the category and it contains all indecomposable objects except $L(\varnothing)$. This is the tensor ideal $\mN$ of negligible morphisms.  Using Lemma~\ref{L:ideals} it follows that there is a unique proper ideal.  That is, $\ideal_{L(\Box )}$ is the unique proper ideal and it equals $\ideal (\mN )$.

\subsection{Dimensions in the Non-Semisimple Block}\label{Dimsec}  By \cite[Theorem 6.4]{CO1} the category $\uRep(S_0)$ has a unique nontrivial block.  The indecomposables in this block can be described explicitly and are denoted by $L_n=L((1^n))$ for $n \in \Z_{\geq 0}$.   In this section we compute the modified dimensions $\md_{\mt}\left( L_n\right)$ for all $n>0$.  

Recall that any indecomposable object in $\uRep(S_0)$ is of the form $([n], e)$ for some primitive idempotent $e\in FP_n(0)$ (see \cite[Proposition 2.20(2)]{CO1}).  Moreover, $A:=([n], e)$ is the direct summand of $B:=([n], \id_n)=L(\Box)\otimes ([n-1], \id_{n-1})$ where the inclusion map $A\to B$ and the projection map $B\to A$ are both given by $e$.  Hence 
\begin{equation*}\label{dcalc}\includegraphics{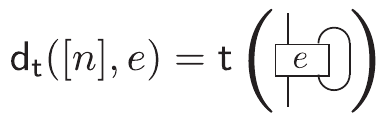}.\end{equation*}

Let $S_{n}$ be the symmetric group on $n$ letters whose elements are viewed as endomorphisms in Deligne's category \cite[Remark 2.14]{CO1}, and let $\sgn : S_{n} \to \{\pm 1 \}$ be the usual sign function.  Recall from \cite[Proposition 6.1]{CO1} that $L_n\cong([n], s_n)$ where 
\begin{equation}\label{E:sndef}
s_n=\frac{1}{n!}\sum_{\sigma\in S_n}\sgn (\sigma)\sigma.
\end{equation}
  In particular, $$\includegraphics{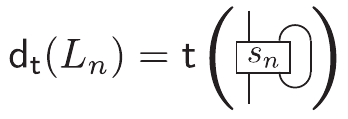}.$$  
For example, $$\includegraphics{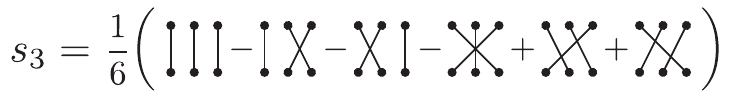}$$ so that $$\includegraphics{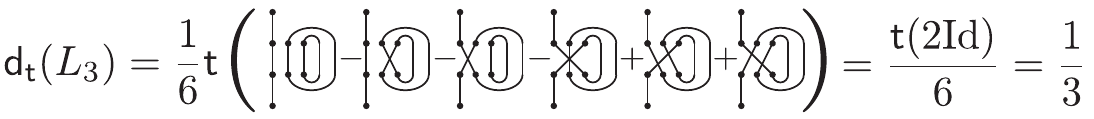}.$$
More generally, given $\sigma\in S_n$, 
$$\includegraphics{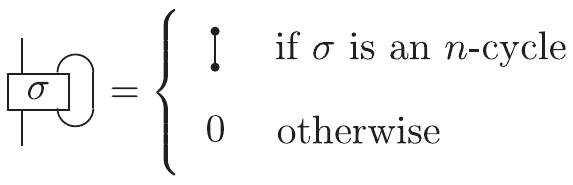}.$$ Hence $$\md_{\mt}(L_n)=(-1)^{n+1}\frac{(\text{number of $n$-cycles in }S_n)}{n!}=\frac{(-1)^{n+1}}{n}.$$

\section{A trace on the ideal of negligibles in $\uRep(S_t)$ when $t\in\Z_{\geq0}$}

We now consider the general case when $t$ is a nonnegative integer.  Let $\mN$ be the tensor ideal of negligible morphisms and let $\ideal = \ideal (\mN )$. Recall that by definition we call the objects of $\ideal$ \emph{negligible}.  In this section we show there exists a nonzero trace on $\ideal$ in $\uRep(S_t)$ when $t$ is a nonnegative integer.

\subsection{Notation}  Given $\sigma\in S_n$ and $I\subset\{1,\ldots, n\}$, let $\sigma_I$ denote the partition diagram obtained from the partition diagram for $\sigma$ by removing all edges adjacent to top vertices labelled by elements of $I$.  Also, for $i\in\{1,\ldots, n\}$ we write $\sigma_i=\sigma_{\{i\}}$.
For example, if 
$$\includegraphics{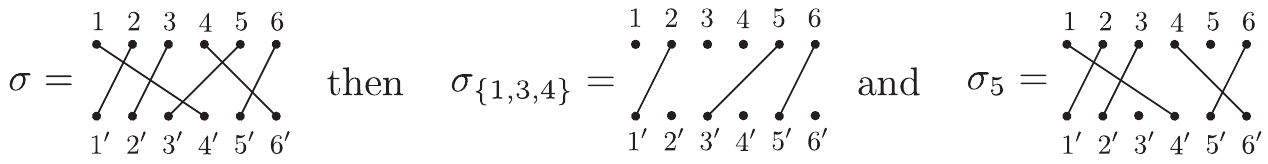}.$$
Given $n\in\mathbb{N}$, write $x_n=(1_{S_n})_n$ where $1_{S_n}$ is the identity permutation in $S_n$.
Finally, let $S_n^-:=\{\sigma_i~|~\sigma\in S_n, 1\leq i\leq n\}$.  

\subsection{The Object $M_n$}\label{SS:Mndef}   Let $M_n:=([n], s_n)\in\uRep(S_t)$ where $s_n$ is as in Equation~\ref{E:sndef}.  

\begin{proposition}\label{decM} We have the following results about $M_{n}$.
\begin{enumerate}
\item $M_n= L((1^n))$ in $\uRep(S_0)$ for all $n\geq 0$.
\item $M_n=L((1^n))\oplus L((1^{n-1}))$ in $\uRep(S_t)$ for all $n>0$ whenever $t\not=0$.
\end{enumerate}
\end{proposition}

\begin{proof} (1) This is \cite[Proposition 6.1]{CO1}.

(2)  Notice $\Lift_t(M_n)=([n], s_n)$ for all $t\in F$ (see \cite[\S3.2]{CO1}).  In particular, $\Lift_t(M_n)=\Lift_0(M_n)$ for all $t\in F$.  Hence, by part (1) along with \cite[Example 5.10(1), Lemma 5.20(2)]{CO1}, $\Lift_t(M_n)=L((1^n))\oplus L((1^{n-1}))$ for all $t\in F$.  
Notice that $P_{(1^n)}(x)=\frac{1}{n!}\prod_{k=1}^n(x-k)$ for all $n>0$ (see \cite[section 3.5]{CO1}).  Hence $L((1^n))$ is in a nontrivial block of $\uRep(S_t)$ if and only if $t$ is a nonnegative integer with $t\not\in\{1,\ldots, n\}$ (see \cite[Proposition 5.11]{CO1}).
If $t$ is an integer greater than $n$, then $L((1^n))$ is the minimal object in a nontrivial block of $\uRep(S_t)$ (see \cite[Corollary 5.9]{CO1}).  Hence, by \cite[Lemma 5.20(1)]{CO1}, $\Lift_t(L((1^n)))=L((1^n))$ for all $n>0$, $t\not=0$.  Therefore, by \cite[Proposition 3.12(3)]{CO1}, $M_n=L((1^n))\oplus L((1^{n-1}))$ in $\uRep(S_t)$ for all $n>0$, $t\not=0$.
\end{proof}

\begin{corollary}\label{Mneg} Suppose $t$ and $n$ are nonnegative integers with $t<n$.  Then $M_n$ is a negligible object in $\uRep(S_t)$.
\end{corollary}

\begin{proof} $M_n$ is negligible if and only if the image of $M_n$ is zero under the functor $\uRep(S_t)\to\Rep(S_t)$ (see for instance \cite[Theorem 3.24]{CO1}).  The result now follows from Proposition \ref{decM} along with \cite[Proposition 3.25]{CO1}.
\end{proof}

In the remainder of this section we examine the endomorphisms of $M_n$.

\begin{proposition}\label{snprop} We have the following equalities.

\begin{enumerate}
\item $\sigma s_n=s_n\sigma=\sgn(\sigma)s_n$ for all $\sigma\in S_n$. 
\item  $s_n\pi s_n=0$ for all partition diagrams $\pi\not\in S_n\sqcup S_n^-$. 
\item  $s_n\sigma_i s_n=\sgn(\sigma)s_nx_n s_n$ for all $\sigma\in S_n, 1\leq i\leq n$.
\end{enumerate}
\end{proposition}

\begin{proof} Part (1) is clear.

(2) If $\pi\not\in S_n\sqcup S_n^-$ then one of the following is true:  (i) two of the top vertices of $\pi$ are in the same part; (ii) two of the top vertices of $\pi$ are in parts of size one; (iii) two of the bottom vertices of $\pi$ are in the same part; (iv) two of the bottom vertices of $\pi$ are in parts of size one.  If (i) or (ii) (resp. (iii) or (iv)) is true, then there exists a transposition $\tau\in S_n$ with $\pi\tau=\pi$ (resp $\tau\pi=\pi$).  By part (1) $\tau s_n$ (resp. $s_n\tau$) is equal to $-s_n$, hence we have   $s_n\pi s_n=s_n\pi\tau s_n=-s_n\pi s_n$ (resp, $s_n\pi s_n=s_n\tau\pi s_n=-s_n\pi s_n$).  The result follows since $F$ is not of characteristic 2.

(3) Suppose $\sigma\in S_n$ and $i\in\{1,\ldots n\}$.  If we let $\tau\in S_n$ denote the transposition $i\leftrightarrow n$, then $\tau\sigma^{-1}\sigma_i\tau=x_n$.  Hence, by part (1),  $s_n\sigma_i s_n=\sgn(\sigma^{-1})s_n\tau\sigma^{-1}\sigma_i\tau s_n=\sgn(\sigma)s_nx_n s_n$.
\end{proof}

\begin{corollary}\label{Mbasis} The set $\{s_n, s_nx_ns_n\}$ is a basis of $\End_{\uRep(S_t)}(M_n)$ for all $t\in F$, $n>0$.
\end{corollary}

\subsection{An Ambidextrous Trace on $M_n$}\label{SS:Mnambi}  Suppose $t\in F$ and $n>0$.   By Corollary~\ref{Mbasis}, to define a linear functional on $\End_{\uRep(S_t)}(M_n)$ it suffices to give the values of the linear functional on $s_n$ and $s_nx_ns_n$.  Let $\mt_n$ be the following linear map:
$$\begin{array}{rcl}
\mt_n:\End_{\uRep(S_t)}(M_n)&\to& F\\
s_n & \mapsto & 1\\
s_nx_ns_n&\mapsto&1\\
\end{array}$$ 
Notice that $\mt_1$ is the ambidextrous trace on $M_1=L(\Box)$ in $\uRep(S_0)$ studied in section \ref{0trace}.    
In this section we will show that $\mt_n$ is an ambidextrous trace on $M_n$ in $\uRep(S_t)$ for all $t\in F$, $n>0$.  To do so, we must examine the endomorphism ring $\End_{\uRep(S_t)}(M_n\otimes M_n)=(s_n\otimes s_n)FP_{2n}(t)(s_n\otimes s_n)$.  

For the remainder of this section assume $n>0$.
Given a partition diagram $\pi\in P_{2n}$, write $\pi_L$ (resp. $\pi_R$) for the partition diagram in $P_n$ obtained by restricting the partition $\pi$ to the vertices $\{1, 1',\ldots, n, n'\}$ (resp. $\{n+1, (n+1)',\ldots, 2n, (2n)'\}$).  
For example, 
$$\includegraphics{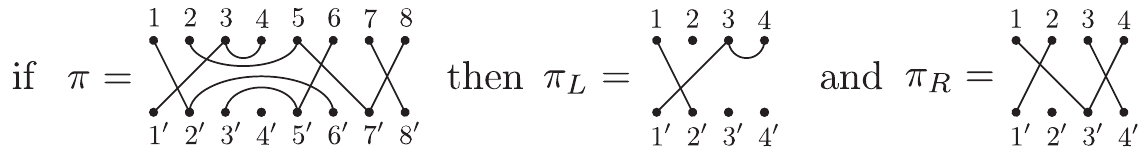}.$$  
Finally, let $\Theta_1, \Theta_2: P_{2n}\to \End_{\uRep(S_t)}(M_n)$ be the maps given by 
$$\includegraphics{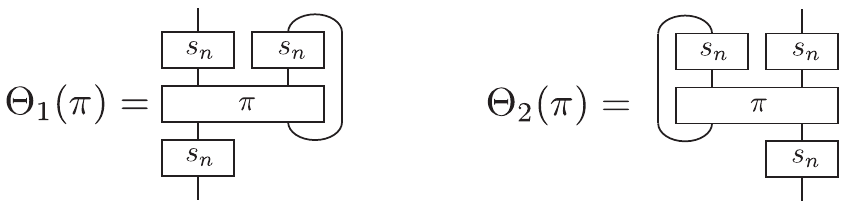}.$$ The following lemma is the first of three lemmas concerning $\Theta_1$ and $\Theta_2$ which will be used to show that $\mt_n$ is an ambidextrous trace.

\begin{lemma}\label{Thzero}  Suppose $\pi\in P_{2n}$ is a partition diagram such that $\pi_L\not\in S_n\sqcup S_n^-$ or $\pi_R\not\in S_n\sqcup S_n^-$.  Then $\Theta_1(\pi)=0=\Theta_2(\pi)$.
\end{lemma}

\begin{proof} If $\pi_L\not\in S_n\sqcup S_n^-$ then (arguing as in the proof of Proposition~\ref{snprop}) there exists a transposition $\tau\in S_n$ with $\pi(\tau\otimes\id_n)=\pi$ or $(\tau\otimes\id_n)\pi=\pi$.  Thus $$\includegraphics{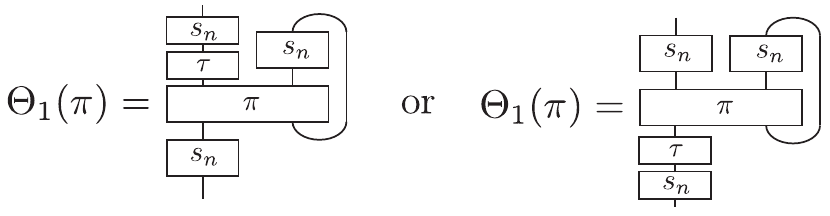}$$ In either case, by Proposition~\ref{snprop}(1), $\Theta_1(\pi)=-\Theta_1(\pi)$ and hence $\Theta_1(\pi)=0$ as the characteristic of $F$ is not 2.  Moreover, $$\includegraphics{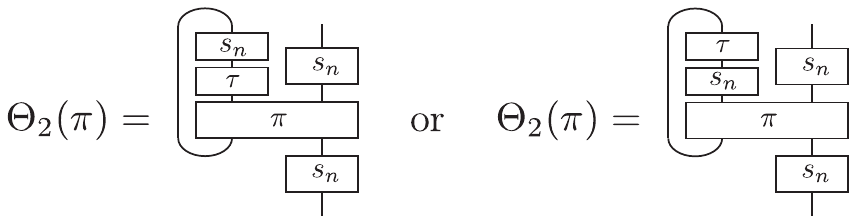}$$ Using Proposition~\ref{snprop}(1) again, we have $\Theta_2(\pi)=-\Theta_2(\pi)$ so that $\Theta_2(\pi)=0$.  The proof when $\pi_R\not\in S_n\sqcup S_n^-$ is similar.
\end{proof}

Now for the second lemma on $\Theta_1$ and $\Theta_2$.  In this lemma, the symbol $\geq$ refers to the partial order on partition diagrams found in \cite[Section 2.1]{CO1}.

\begin{lemma}\label{lemmaSTP} Suppose $\pi\in P_{2n}$ is such that $\pi_L, \pi_R\in S_n\sqcup S_n^-$.  If $\pi_L$ (resp. $\pi_R$) is in $S_n^-$ then there exists a partition diagram $\pi'\in P_{2n}$ with $\pi'_L$ (resp. $\pi'_R$) in $S_n$ such that $\pi'_R\geq\pi_R$ (resp. $\pi'_L\geq\pi_L$) and $\mt_n(\Theta_i(\pi'))=\mt_n(\Theta_i(\pi))$ for $i=1,2$.
\end{lemma}

\begin{proof} Suppose $\pi_L\in S_n^-$ so that $\pi_L=\sigma_i$ for some $\sigma\in S_n$ and $i\in\{1,\ldots, n\}$.  Let $\pi'\in P_{2n}$ be the partition diagram obtained from $\pi$ by adding an edge between the vertices labelled $i$ and $\sigma(i)'$.  Then $\pi'\geq \pi$ which implies   $\pi'_R\geq \pi_R$.  Also, $\pi'_L=\sigma\in S_n$.  Moreover, given $\tau\in S_n$, the connected components of the graph \begin{equation}\label{tdp}\includegraphics{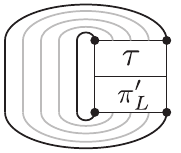}\end{equation} are all cycles.  Hence, as the graph of \begin{equation}\label{td}\includegraphics{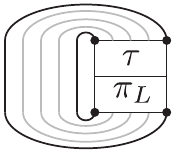}\end{equation} is obtained from (\ref{tdp}) by deleting one edge, the partitions of $\{1, 1',\ldots, n, n'\}$ corresponding to the 
connected components of (\ref{tdp}) and (\ref{td}) are equal.  Therefore, $$\includegraphics{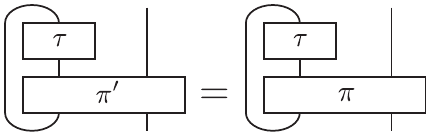}$$ for every $\tau\in S_n$.  Hence $$\includegraphics{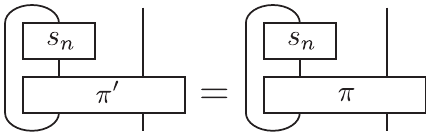}$$ which implies $\Theta_2(\pi')=\Theta_2(\pi)$.  

Now, fix $\rho\in S_n$ and let $\mu_\rho, \mu_\rho'\in P_n$, $\ell(\rho), \ell(\rho)'\in\Z_{\geq0}$ be such that $$\includegraphics{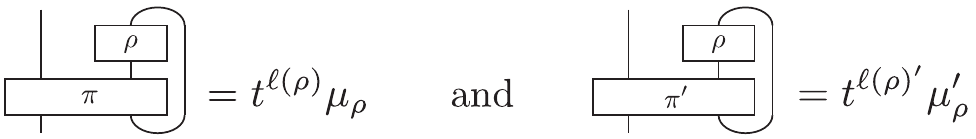}.$$  Notice that $\ell(\rho)$ (resp. $\ell(\rho)'$) is the number of connected components of the partition diagram $\pi(\id_n\otimes \rho)$ (resp. $\pi'(\id_n\otimes \rho)$) which only contain vertices labelled by integers greater than $n$.  The connected components of $\pi$ and $\pi'$ (and hence of  $\pi(\id_n\otimes \rho)$ and $\pi'(\id_n\otimes \rho)$) which only contain vertices labelled by integers greater than $n$ are identical.  Hence $\ell(\rho)=\ell(\rho)'$.  Also, it is easy to see that $\mu_\rho\geq\pi_L$.  Thus, as $\pi_L\in S_n^-$, there are three cases:  (i) $\mu_\rho=\pi_L$, (ii) $\mu_\rho=\pi'_L$, (iii) $\mu_\rho\not\in S_n\sqcup S_n^-$.  Next, we show that $\mt_n(s_n\mu_\rho s_n)=\mt_n(s_n\mu'_\rho s_n)$ in each of the three cases above:

\noindent (i) If $\mu_\rho=\pi_L=\sigma_i$ then it is easy to see that $\mu'_\rho=\pi'_L=\sigma$.  Hence, by Proposition~\ref{snprop}(1)\&(3) and the definition of $\mt_n$, $\mt_n(s_n\mu_\rho s_n)=\sgn(\sigma)=\mt_n(s_n\mu'_\rho s_n)$.

\noindent (ii) If $\mu_\rho=\pi'_L$ then $\mu'_\rho=\pi_L'$ too.  Hence $\mt_n(s_n\mu_\rho s_n)=\mt_n(s_n\mu'_\rho s_n)$.

\noindent (iii) If $\mu_\rho\not\in S_n\sqcup S_n^-$ then $\mu'_\rho\not\in S_n\sqcup S_n^-$ too.  Therefore, by Proposition~\ref{snprop}(2), $\mt_n(s_n\mu_\rho s_n)=0=\mt_n(s_n\mu'_\rho s_n)$.

As $\rho$ was an arbitrary element of $S_n$, we have  $$\mt_n(\Theta_1(\pi))=\sum_{\rho\in S_n}\sgn(\rho)t^{\ell(\rho)}\mt_n(s_n\mu_\rho s_n)=\sum_{\rho\in S_n}\sgn(\rho)t^{\ell(\rho)'}\mt_n(s_n\mu'_\rho s_n)=\mt_n(\Theta_1(\pi')).$$
The statement of the lemma with $\pi_L\in S_n^-$ follows.  The proof when $\pi_R\in S_n^-$ is similar.
\end{proof}  

Before proving the third and final lemma on $\Theta_1$ and $\Theta_2$ we need to introduce a bit more notation.  
Suppose $\pi\in P_{2n}$ is such that $\pi_L, \pi_R\in S_n$.  Let $I=I_\pi\subset\{1,\ldots, n\}$ denote the set of all $i\in\{1,\ldots, n\}$ which correspond to vertices in $\pi$ whose parts are of size two.  Now let $\pi_{L-R}=\sigma_I$ where $\sigma\in S_n$ is any permutation with $\sigma(i)=j$ whenever the top vertices labelled by $i$ and $n+j$ are in the same part of $\pi$.  For example, 
$$\includegraphics{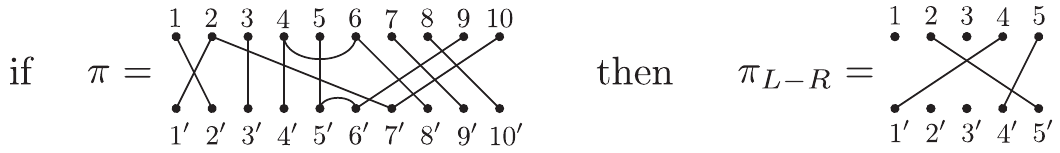}.$$
The following proposition concerning $\pi_{L-R}$ will be used in the proof of the final lemma on $\Theta_1$ and $\Theta_2$.

\begin{proposition}\label{piprop} Suppose $\pi\in P_{2n}$ is such that $\pi_L, \pi_R\in S_n$.  If $\pi'\in P_{2n}$ has $\pi'_L=\pi_L$, $\pi'_R=\pi_R$, and $\pi'_{L-R}=\pi_{L-R}$, then $\pi'=\pi$.
\end{proposition}

\begin{proof} Since $\pi_L, \pi_R\in S_n$, each part of $\pi$ is of one of the following three types:  (i) $\{i, j'\}$ with $1\leq i, j\leq n$; (ii) $\{i, j'\}$ with $n<i, j\leq 2n$; (iii) $\{i, j', k, l'\}$ with $1\leq i, j\leq n$ and $n<k, l\leq 2n$.  Hence $\pi$ is completely determined by $\pi_L, \pi_R$, and $\pi_{L-R}$.
\end{proof}

\begin{lemma}\label{lemmaSn} If $\pi\in P_{2n}$ is such that $\pi_L, \pi_R\in S_n$, then $\Theta_1(\pi)=\Theta_2(\pi)$.
\end{lemma}

\begin{proof} Write $I=I_\pi$ as above and let $\sigma\in S_n$ be any permutation with $\pi_{L-R}=\sigma_I$.  Let $\mu, \mu'\in P_{2n}$ be the following partition diagrams
$$\includegraphics{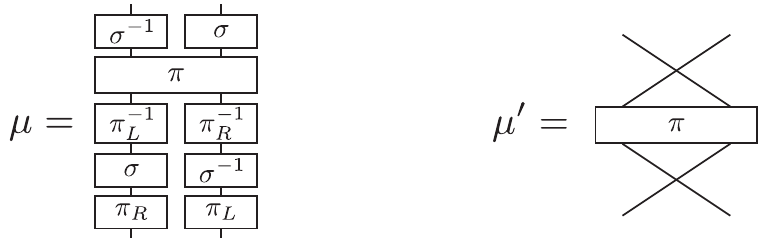}.$$ 
It is easy to check that $\mu_L=\pi_R=\mu'_L$, $\mu_R=\pi_L=\mu'_R$, and $\mu_{L-R}=(\sigma^{-1})_{J}=\mu'_{L-R}$ where $J=\{\sigma(i)~|~i\in I\}$.  Hence, by Proposition~\ref{piprop}, $\mu=\mu'$.  Also, by Proposition~\ref{snprop}(1), $\Theta_1(\pi)=\Theta_1(\mu)$. Therefore $\Theta_1(\pi)=\Theta_1(\mu)=\Theta_1(\mu')=\Theta_2(\pi)$.
\end{proof}

Now we prove the main result of this section.

\begin{theorem}\label{amb} $\mt_n$ is a nonzero ambidextrous trace on $M_n$ in $\uRep(S_t)$ for all $t\in F$, $n>0$.
\end{theorem}

\begin{proof} We are required to show $\mt_n(\Theta_1(\pi))=\mt_n(\Theta_2(\pi))$ for all $\pi\in P_{2n}$.  If either $\pi_L$ or $\pi_R$ is not in $S_n\sqcup S_n^-$ then the result follows from Lemma~\ref{Thzero}.  Hence we can assume $\pi_L, \pi_R\in S_n\sqcup S_n^-$.  If $\pi_L\in S_n^-$ then by Lemma~\ref{lemmaSTP} there exists $\pi'\in P_{2n}$ with $\pi'_L\in S_n$ and $\mt_n(\Theta_i(\pi'))=\mt_n(\Theta_i(\pi))$ for $i=1, 2$.  If $\pi'_R\not\in S_n\sqcup S_n^-$ then by Lemma~\ref{Thzero} we have $\mt_n(\Theta_i(\pi))=\mt_n(\Theta_i(\pi'))=0$ for $i=1, 2$; hence we can assume $\pi'_R\in S_n\sqcup S_n^-$.  If $\pi'_R\in S_n^-$ then by Lemma~\ref{lemmaSTP} there exists $\pi''\in P_{2n}$ with $\pi''_R\in S_n$, $\pi''_L\geq \pi'_L$, and $\mt_n(\Theta_i(\pi''))=\mt_n(\Theta_i(\pi'))$ for $i=1, 2$.  If $\pi''_L\not\in S_n\sqcup S_n^-$ then by Lemma~\ref{Thzero} we have $\mt_n(\Theta_i(\pi))=\mt_n(\Theta_i(\pi''))=0$ for $i=1, 2$; hence we can assume $\pi''_L\in S_n\sqcup S_n^-$.  Also, since $\pi''_L\geq \pi'_L$ and $\pi'_L\in S_n$ it follows that $\pi''_L\not\in S_n^-$.  Thus, we can assume $\pi''_L\in S_n$.  In this case, by Lemma~\ref{lemmaSn}, $\mt_n(\Theta_1(\pi))=\mt_n(\Theta_1(\pi''))=\mt_n(\Theta_2(\pi''))=\mt_n(\Theta_2(\pi)).$
\end{proof}

\begin{corollary}\label{C:nonzerotrace} If $t\in\Z_{\geq 0}$ then there exists a nonzero trace on the ideal of negligible objects in $\uRep(S_t)$.
\end{corollary}

\begin{proof} This follows from Theorem~\ref{T:ambitrace} and Theorem \ref{amb} that there is a nonzero trace on $\ideal_{M_{n}}$. By Corollary~\ref{Mneg} we have that $M_{n}$ is an object in $\ideal = \ideal (\mN )$ and, hence, $\ideal_{M_{n}} \subseteq \ideal$.  By Lemma~\ref{L:ideals} we have $J(\ideal_{M_{n}})$ is not the zero ideal and $J(\ideal_{M_{n}}) \subseteq J(\ideal)$.  By the classification of tensor ideals in $\uRep(S_t)$, $\mN$ is the unique proper tensor ideal of $\uRep(S_t)$.  Thus $J(\ideal_{M_{n}}) = J(\ideal)$ and by Lemma~\ref{L:ideals} $\ideal_{M_{n}}=\ideal$.  This proves the trace function is in fact defined on the entire ideal of negligible objects in $\uRep(S_t)$.
\end{proof}
\begin{remark}\label{R:SemisimpleCase}  When $t \notin \Z_{\geq 0}$ then $\uRep (S_{t})$ is a semisimple category \cite[Th\'eor\`eme 2.18]{Del07}.  Consequently, there are no proper ideals in $\uRep (S_{t})$ and the categorical trace is the only nontrivial trace. 
\end{remark}

\section{A Graded Variation on Deligne's Category}\label{S:quantumdouble} In this section we briefly examine a graded version of Deligne's category and show that it can be used to recover the writhe of a knot -- a well-known invariant of framed knots.

\subsection{}\label{SS:gradedvariant}   

Fix $t,q \in F $ with $q \neq 0$.  We then let $\gruRep_{0}(S_{t})=\gruRep_{0}(S_{t})_{q}$ be the category defined as follows.  The objects are all pairs $[a,b]$ for all $a,b \in \Z_{\geq 0}$.  We put a $\Z$-grading on the objects of the category by setting the degree of $[a,b]$ to be $a-b$.   

The morphisms are given by 
\[
\Hom_{\gruRep_{0}(S_{t})}\left([a,b],[c,d] \right) = \begin{cases}  FP_{a+b,c+d} , &\text{when $a-b=c-d$}; \\
0, &\text{else}.
\end{cases}
\]  The composition of morphisms is given by the same vertical concatenation of diagrams rule as in the definition of Deligne's category $\uRep_{0}(S_{t})$.  By definition the morphisms preserve the $\Z$-grading.

Define the tensor product on $\gruRep_{0}(S_{t})$ by 
\[
[a,b] \otimes [c,d] = [a+c, b+d]
\] and on morphisms by horizontal concatenation of diagrams just as in Deligne's category $\uRep_{0}(S_{t})$.   The associativity constraint is given by the identity.  

The unit object is then $\unit = [0,0]$ and the unit constraints $[0,0] \otimes [a,b] \to [a,b]$ and $[a,b] \otimes [0,0] \to [a,b]$ are given by the identity.

The dual of the object $[a,b]$ is given by
\[
[a,b]^{*}=[b,a].
\]  The evaluation morphism $\ev : [a,b]^{*} \otimes [a,b] \to \unit$ is given by a diagram in $FP_{2a+2b,0}$ which gives the evaluation morphism $[a+b]^{*} \otimes [a+b] \to \unit$ in $\uRep_{0}(S_{t})$.  Similarly, the coevaluation morphism $\coev : \unit  \to [a,b] \otimes [a,b]^{*}$ is given by the coevaluation in $\uRep_{0}(S_{t})$.

Let $\beta_{n,m}: [n] \otimes [m] \to [m] \otimes [n]$ be the diagram in $FP_{m+n, m+n}$ which gives the braiding in $\uRep_{0}(S_{t})$.  The braiding on $\gruRep_{0}(S_{t})$,
\[
c_{[a,b],[c,d]}: [a,b] \otimes [c,d] \to [c,d] \otimes [a,b],
\] is then given by 
\[
c_{[a,b],[c,d]} = q^{(a-b)(c-d)}\beta_{a+b,c+d}.
\]  The fact that this gives a braiding on the category follows from the calculation which shows that the $\beta$'s define a braiding in $\uRep_{0}(S_{t})$ along with the fact that morphisms in $\gruRep_{0}(S_{t})$ are grading preserving.  Also note that this braiding is usually \emph{not} symmetric.

Finally, the twist morphisms $\theta_{[a,b]} : [a,b] \to [a,b]$ are given by 
\[
\theta_{[a,b]}= q^{(a-b)^{2}}\Id_{a+b},
\] where $\Id_{a+b}$ is the identity in $FP_{a+b}$.  

A direct verification of the axioms shows that the above tensor product, unit, duality, braiding, and twists make $\gruRep_{0}(S_{t})$ into a ribbon category.  

Write $\gruRep (S_{t})$ for the Karoubian envelope of the additive envelope of $\gruRep_{0}(S_{t})$.  
The ribbon category structure on  $\gruRep_{0}(S_{t})$ defines a ribbon category structure on $\gruRep (S_{t})$ just as it does going from $\uRep_{0}(S_{t})$ to $\uRep (S_{t})$.  We also note that using the definition of the additive and Karoubian envelopes we see that the category $\gruRep (S_{t})$ inherits a $\Z$-grading and all morphisms are grading preserving. 

We have the following ``degrading'' functor between the graded and ungraded versions of Deligne's category.

\begin{proposition}\label{P:RelatingCats} Let $q \in  F \backslash \{0 \}$.
 Then there is a faithful functor 
\[
\F: \gruRep(S_{t})_{q} \to \uRep(S_{t}).
\]

This functor is induced by the functor
\[
\F_{0}: \gruRep_{0}(S_{t})_{q} \to \uRep_{0}(S_{t})
\] given by setting
\[
\F_{0}([a,b])=[a+b]
\]
for all $a,b \in \Z_{\geq 0}$.  On morphisms, $\F_{0}$ is the identity; that is, for 
\[
f \in \Hom_{\gruRep_{0}(S_{t})}([a,b], [c,d]) \subseteq FP_{a+b,c+d},
\]
we set 
\[
\F_{0}(f) = f: [a+b] \to [c+d].
\] 

\end{proposition}

\begin{proof} The construction of the additive and Karoubian envelopes shows that $\F_{0}$ induces a functor $\F:\gruRep(S_{t})_{q} \to \uRep(S_{t})$. The statement about injectivity on morphisms follows from the fact that $\F_{0}$ is injective on morphisms and the construction of $\F$.
\end{proof}

\begin{remark}\label{R:degrading}  It is straightforward to verify that $\F$ is a tensor functor (i.e.\ $\F(X \otimes Y) = \F(X) \otimes \F(Y)$ for all objects $X$ and $Y$, $\F(\unit ) = \unit$, and preserves the associativity and unit constraints) and that it preserves duals (i.e.\ $\F(X^{*})=\F(X)^{*}$ for all objects $X$ and preserves the evaluation and coevaluation morphisms).  If $V,W$ are homogeneous objects in $\gruRep(S_{t})_{q}$ with $V$ of degree $r \in \Z$ and $W$ of degree $s \in \Z$, then 
\[
\F(c_{V,W}) = q^{r  s} c_{\F(V), \F(W)},
\] where $c_{\F(V), \F(W)}$ is the braiding in the category $\uRep(S_{t})$.  Similarly,
\[
\F(\theta_{V}) = q^{r^{2}}\theta_{\F(V)},
\] where $\theta_{\F(V)}$ is the twist in the category $\uRep(S_{t})$.

Thus the functor $\F$ preserves braidings and twists if and only if $q=1$ but, as we will soon see, it is close enough for our purposes.
\end{remark}

\begin{theorem}\label{T:gradedambis}  Let $V$ be an object of $\gruRep(S_{t})_{q}$ such that $\F(V)$ admits an ambidextrous trace 
\[
\mt_{\F(V)}: \End_{\uRep(S_{t})}(\F(V)) \to F .
\]  Then the map 
\[
\mt: \End_{\gruRep(S_{t})_{q}}(V) \to F 
\] given by $g \mapsto \mt_{\F(V)}(\F(g))$ defines an ambidextrous trace on $V$.
\end{theorem}

\begin{proof} Since both $\F$ and $\mt_{\F(V)}$ are linear, the map $\mt$ is linear and so without loss of generality we may assume $V$ is homogenous of degree $d \in \Z$.  Let $h \in \End_{\gruRep(S_{t})_{q}}(V \otimes V)$ and consider the morphism $\Tr_{R}(h): V \to V$.  Since $\F$ is a tensor functor, preserves evaluation and coevaluation, and takes the braiding and twist to a $q$ multiple of the braiding and twist, it follows that $\F(\Tr_{R}(h))$ is a $q$-multiple of $\Tr_{R}(\F(h))$.  A calculation using Remark~\ref{R:degrading} shows that in fact $\F(\Tr_{R}(h))=\Tr_{R}(\F(h))$. Similarly, $\F(\Tr_{L}(h))=\Tr_{L}(\F(h))$.  From this it is immediate that $\mt$ defines an ambidextrous trace.
\end{proof}

\begin{remark}\label{R:gradedambis}  Fix a nonnegative integer $n$ and fix $a,b \in \Z_{\geq 0}$ such that $a+b = n$.  Let $s_{n} \in FP_{n,n}$ be as in Equation \ref{E:sndef}.  Then $M_{a,b}:=([a,b], s_{n})$ is an object of $\gruRep(S_{t})_{q}$ and $\F(M_{a,b})=M_{n}$, the object of $\uRep(S_{t})$ defined in Section~\ref{SS:Mndef} and shown to admit an ambidextrous trace in Section~\ref{SS:Mnambi}.  By the previous theorem, the object $M_{a,b}$ in  $\gruRep(S_{t})_{q}$ admits an ambidextrous trace.  In particular, say we fix $a,b$ so that $a-b \neq 0$ and fix $q \in F$ not a root of unity\footnote{Otherwise we obtain the writhe modulo the order of $q$.}.  Then  $M_{a,b}$ is homogeneous of degree $a-b$ and it is not difficult to see that if $K$ is an oriented framed knot, then the invariant obtained by labeling $K$ by $M_{a,b}$ is the function 
\[
K \mapsto q^{(a-b)^{2}\omega},
\] where $\omega \in \Z$ is the \emph{writhe} of $K$.

\end{remark}

\renewcommand{\bibname}{\textsc{references}} 
\bibliographystyle{amsalpha}	
	\addcontentsline{toc}{chapter}{\bibname}
	\bibliography{references}		

\vfill\vfill\vfill

\end{document}